\documentclass{llncs}
\usepackage{version}
\usepackage{tsf}

\renewcommand{\vec}[1]{#1}

\title{Transformation of Fractions into \\ Simple Fractions in Divisive
       Meadows}
\author{J.A. Bergstra \and C.A. Middelburg}
\institute{Informatics Institute, Faculty of Science, University of
           Amsterdam, \\
           Science Park~904, 1098~XH Amsterdam, the Netherlands \\
           \email{J.A.Bergstra@uva.nl,C.A.Middelburg@uva.nl}}

\begin{document}
\maketitle

\begin{abstract}
Meadows are alternatives for fields with a purely equational 
axiomatization.
At the basis of meadows lies the decision to make the multiplicative 
inverse operation total by imposing that the multiplicative inverse of
zero is zero.
Divisive meadows are meadows with the multiplicative inverse operation
replaced by a division operation.
Viewing a fraction as a term over the signature of divisive meadows that
is of the form $p \mdiv q$, we investigate which divisive meadows admit 
transformation of fractions into simple fractions, i.e.\ fractions 
without proper subterms that are fractions.
\begin{classcode}
12E12, 12L12, 68Q65
\end{classcode}
\end{abstract}

\section{Introduction}
\label{sect-intro}

To our knowledge, all existing definitions of a fraction are 
insufficiently precise to allow the validity of many non-trivial 
statements about fractions to be established.
The work presented in this paper is concerned with the rigorous 
definition of a fraction and the validity of statements related to the
question whether each fraction can be transformed into a simple fraction 
(colloquially described as a fraction where neither the numerator nor 
the denominator contains a fraction).
This work is carried out in the setting of divisive meadows.

Because fields do not have a purely equational axiomatization, the 
axioms of a field cannot be used in applications of the theory of 
abstract data types to number systems based on rational, real or complex 
numbers.
In~\cite{BT07a}, meadows are proposed as alternatives for fields with a
purely equational axiomatization.
At the basis of meadows lies the decision to make the multiplicative 
inverse operation total by imposing that the multiplicative inverse of 
zero is zero. 
A meadow is a commutative ring with a multiplicative identity element
and a total multiplicative inverse operation satisfying the two 
equations $(x\minv)\minv = x$ and $x \mmul (x \mmul x\minv) = x$.
It follows from the axioms of a meadow that the multiplicative inverse 
operation also satisfies the equation $0\minv = 0$.
All fields in which the multiplicative inverse of zero is zero, called
zero-totalized fields, are meadows, but not conversely.
Because of their purely equational axiomatization, all meadows are total 
algebras and the class of all meadows is a variety.

In~\cite{BM09g}, divisive meadows are proposed.
A divisive meadow is a commutative ring with a multiplicative identity
element and a total division operation satisfying the three equations 
$1 \mdiv (1 \mdiv x) = x$, $(x \mmul x) \mdiv x = x$, and 
$x \mdiv y = x \mmul (1 \mdiv y)$.
It follows from the axioms of a divisive meadow that the division 
operation also satisfies the equation $x \mdiv 0 = 0$.
We coined the alternative name inversive meadow for a meadow.
The equational axiomatizations of inversive meadows and divisive meadows
are essentially the same in the sense that they are definitionally 
equivalent.

We expect the zero-totalized multiplicative inverse and division 
operations of inversive and divisive meadows, which are conceptually and 
technically simpler than the conventional partial multiplicative inverse 
and division operations, to be useful in among other things mathematics 
education.
We further believe that viewing fractions as terms over the signature of 
divisive meadows whose outermost operator is the division operator 
gives a rigorous definition of a fraction that can serve as a basis of a
workable theory about fractions for teaching purposes at all levels of 
education (cf.~\cite{Ber15a}).
Divisive meadows are more convenient than inversive meadows for the
definition of fractions because, unlike the signature of inversive 
meadows, the signature of divisive meadows includes the division 
operator.

Viewing fractions as described above has two salient consequences: 
(i)~fractions may contain variables and 
(ii)~fractions may be interpreted in different divisive meadows. 
These consequences lead to the need to make a distinction between  
fractions and closed fractions and to consider properties of fractions 
relative to a particular divisive meadow.
Viewing fractions as described above, many properties of fractions 
considered in the past turn out to be properties of closed fractions 
and/or properties of fractions relative to the divisive meadow of 
rational numbers.

For example, it is known from earlier work on meadows that closed 
fractions can be transformed into simple fractions, i.e.\ fractions 
without proper subterms that are fractions, if fractions are interpreted 
in the divisive meadow of rational numbers.
Now the question arises whether the restriction to closed fractions can 
be dropped and whether this result goes through if fractions are 
interpreted in divisive meadows different from the divisive meadow of 
rational numbers.

In this paper, we investigate which divisive meadows admit 
transformation into simple fractions (for both the general case and the 
case of closed fractions).
Some exemplary results are:
(i)~every model of the axioms of divisive meadow with a finite carrier 
admits transformation into simple fractions;
(ii)~every minimal model of the axioms of a divisive meadow with an 
infinite carrier does not admit transformation into simple fractions;
(iii)~the divisive meadow of rational numbers is the only minimal model 
of the axioms of a divisive meadow with an infinite carrier that admits
transformation into simple fractions for closed fractions.

This paper is organized as follows.
First, we give a survey of inversive meadows and divisive meadows which
includes the signature and axioms for them, general results about them, 
and terminology used in the setting of meadows (Section~\ref{sect-Md}).
Next, we give the definitions concerning fractions and polynomials on 
which subsequent sections are based 
(Section~\ref{sect-basic-defs-fractions})
and establish some auxiliary results concerning divisive meadows which
will be used in subsequent sections 
(Section~\ref{sect-auxiliary-results}).
Then, we establish results about the transformation into simple 
fractions 
(Sections~\ref{sect-fractions} and~\ref{sect-closed-fractions}).
Following this, we establish results that are related to the results in 
the preceding sections, but do not concern fractions 
(Section~\ref{sect-miscellaneous}).
Finally, we make some concluding remarks (Section~\ref{sect-concl}).

We conclude this introduction with a corrective note on a remark made 
in~\cite{BM09g} and later papers on meadows.
Skew meadows, which differ from meadows only in that their 
multiplication is not required to be commutative, were already studied  
in~\cite{Kom75a,Kom76a}, where they go by the name of desirable 
pseudo-fields.
In 2009, we first read about desirable pseudo-fields in~\cite{Ono83a} 
and reported on it in~\cite{BM09g}.
However, we thought incorrectly at the time that desirable pseudo-fields
were meadows.
Hence, we accidentally mentioned that meadows were already introduced 
in~\cite{Kom75a}. 

\section{Inversive Meadows and Divisive Meadows}
\label{sect-Md}

In this section, we survey both inversive meadows and divisive meadows.
Inversive meadows have been proposed as alternatives for fields with a 
purely equational axiomatization in~\cite{BT07a}.
Inversive meadows have been further investigated in 
e.g.~\cite{BHT09a,BR08a,BBP13a,BM14b,BRS09a} and applied in 
e.g.~\cite{BPZ07a,BB09b,BM09d}.
Divisive meadows, which are inversive meadows with the multiplicative 
inverse operation replaced by a division operation, have been proposed 
in~\cite{BM09g}.%
\footnote
{An overview of all the work on meadows done to date and some open 
 questions originating from that work can be found on~\cite{siteMd}.}
In subsequent sections, only divisive meadows are needed.
However, established results about inverse meadows will be used in 
proofs wherever it is justified by the definitional equivalence of the
axiomatizations of inversive meadows and divisive meadows (see also the 
remarks following Theorem~\ref{theorem-defeqv-iMd-dMd} below).

An inversive meadow is a commutative ring with a multiplicative identity
element and a total multiplicative inverse operation satisfying two
equations which imply that the multiplicative inverse of zero is zero.
A divisive meadow is a commutative ring with a multiplicative identity
element and a total division operation satisfying three equations which
imply that division by zero always yields zero.
Hence, the signature of both inversive and divisive meadows include the
signature of a commutative ring with a multiplicative identity element.

The signature of commutative rings with a multiplicative identity
element consists of the following constants and operators:
\begin{itemize}
\item
the \emph{additive identity} constant $0$;
\item
the \emph{multiplicative identity} constant $1$; 
\pagebreak[2]
\item
the binary \emph{addition} operator ${} + {}$; 
\item
the binary \emph{multiplication} operator ${} \mmul {}$;
\item
the unary \emph{additive inverse} operator $- {}$;
\end{itemize}
The signature of inversive meadows consists of the constants and
operators from the signature of commutative rings with a multiplicative
identity element and in addition:
\begin{itemize}
\item
the unary \emph{multiplicative inverse} operator ${}\minv$.
\end{itemize}
The signature of divisive meadows consists of the constants and
operators from the signature of commutative rings with a multiplicative
identity element and in addition:
\begin{itemize}
\item
the binary \emph{division} operator ${} \mdiv {}$.
\end{itemize}
We write:
\begin{ldispl}
\begin{array}{@{}l@{\;}c@{\;}l@{}}
\sigcr    & \mathrm{for} & \set{0,1,{} + {},{} \mmul {}, - {}}\;,
\\
\sigimd   & \mathrm{for} & \sigcr \union \set{{}\minv}\;,
\\
\sigdmd   & \mathrm{for} & \sigcr \union \set{{} \mdiv {}}\;.
\end{array}
\end{ldispl}

We assume that there are infinitely many variables, including $x$, $y$
and $z$.
Terms are build as usual.
We use infix notation for the binary operators, prefix notation for the
unary operator $- {}$, and postfix notation for the unary 
operator~${}\minv$.
We use the usual precedence convention to reduce the need for
parentheses.
Subtraction is introduced as an abbreviation as follows: $p - q$ 
abbreviates $p + (-q)$. 
We use the notation $p^n$ for exponentiation with natural number 
exponents.
For each term $p$ over the signature of inversive meadows or the 
signature of divisive meadows and each natural number $n$, the term 
$p^n$ is defined by induction on $n$ as follows: $p^0 = 1$ and
$p^{n+1} = p^n \mmul p$.
We use the notation $\ul{n}$ for the numeral of $n$.
For each natural number $n$, the term $\ul{n}$ is defined by induction 
on $n$ as follows: $\ul{0} = 0$ and $\ul{n+1} = \ul{n} + 1$.
For convenience, we extend the notation $\ul{n}$ from natural numbers to 
integers by stipulating that $\ul{-n} = - \ul{n}$.

The constants and operators from the signatures of inversive meadows and
divisive meadows are adopted from rational arithmetic, which gives an
appropriate intuition about these constants and operators.

A commutative ring with a multiplicative identity element is a total 
algebra over the signature $\sigcr$ that satisfies the equations given 
in Table~\ref{eqns-commutative-ring}.%
\begin{table}[!t]
\caption
{Axioms of a commutative ring with a multiplicative identity element}
\label{eqns-commutative-ring}
\begin{eqntbl}
\begin{eqncol}
(x + y) + z = x + (y + z)                                             \\
x + y = y + x                                                         \\
x + 0 = x                                                             \\
x + (-x) = 0
\end{eqncol}
\qquad\quad
\begin{eqncol}
(x \mmul y) \mmul z = x \mmul (y \mmul z)                             \\
x \mmul y = y \mmul x                                                 \\
x \mmul 1 = x                                                         \\
x \mmul (y + z) = x \mmul y + x \mmul z
\end{eqncol}
\end{eqntbl}
\end{table}
An \emph{inversive meadow} is a total algebra over the signature 
$\sigimd$ that satisfies the equations given in 
Tables~\ref{eqns-commutative-ring}
and~\ref{eqns-add-inversive-meadow}.%
\begin{table}[!t]
\caption{Additional axioms for an inversive meadow}
\label{eqns-add-inversive-meadow}
\begin{eqntbl}
\begin{eqncol}
{} \\[-3ex]
(x\minv)\minv = x                                       \\
x \mmul (x \mmul x\minv) = x                           
\end{eqncol}
\end{eqntbl}
\end{table}
A \emph{divisive meadow} is a total algebra over the signature $\sigdmd$ 
that satisfies the equations given in Tables~\ref{eqns-commutative-ring}
and~\ref{eqns-add-divisive-meadow}.%
\begin{table}[!t]
\caption{Additional axioms for a divisive meadow}
\label{eqns-add-divisive-meadow}
\begin{eqntbl}
\begin{eqncol}
1 \mdiv (1 \mdiv x) = x                                               \\
(x \mmul x) \mdiv x = x                                               \\
x \mdiv y = x \mmul (1 \mdiv y)
\end{eqncol}
\end{eqntbl}
\end{table}

We write:
\begin{ldispl}
\begin{array}{@{}l@{\;}c@{\;}l@{}}
\eqnscr  &
\multicolumn{2}{@{}l@{}}
 {\mathrm{for\; the\; set\; of\; all\; equations\; in\; Table\;
          \ref{eqns-commutative-ring}}\;,}
\\
\eqnsinv  &
\multicolumn{2}{@{}l@{}}
 {\mathrm{for\; the\; set\; of\; all\; equations\; in\; Table\;
          \ref{eqns-add-inversive-meadow}}\;,}
\\
\eqnsdiv  &
\multicolumn{2}{@{}l@{}}
 {\mathrm{for\; the\; set\; of\; all\; equations\; in\; Table\;
          \ref{eqns-add-divisive-meadow}}\;,}
\\
\eqnsimd & \mathrm{for} & \eqnscr \union \eqnsinv\;,
\\
\eqnsdmd & \mathrm{for} & \eqnscr \union \eqnsdiv\;.
\end{array}
\end{ldispl}

Equations making the nature of the multiplicative inverse operation
in inversive meadows more clear are derivable from the equations 
$\eqnsimd$.
\begin{proposition}
\label{prop-iMd-derivable}
The equations 
\begin{ldispl}
0\minv = 0\;, \quad
1\minv = 1\;, \quad
(- x)\minv = - (x\minv)\;, \quad
(x \mmul y)\minv = x\minv \mmul y\minv 
\end{ldispl}%
are derivable from the equations $\eqnsimd$.
\end{proposition}
\begin{proof}
Theorem~2.2 from~\cite{BT07a} is concerned with the derivability of the 
first equation and Proposition~2.8 from~\cite{BHT09a} is concerned with 
the derivability of the last two equations.
The derivability of the second equation is trivial.
\qed
\end{proof}

The advantage of working with a total multiplicative inverse operation
or a total division operation lies in the fact that conditions such as
$x \neq 0$ in $x \neq 0 \Implies x \mmul x\minv = 1$ are not needed to 
guarantee meaning.

An inversive or divisive meadow is \emph{non-trivial} if it satisfies the 
\emph{separation axiom} 
\begin{ldispl}
0 \neq 1\;;
\end{ldispl}%
and it is an inversive or divisive \emph{cancellation meadow} if it 
satisfies the \emph{cancellation axiom} 
\begin{ldispl}
x \neq 0 \And x \mmul y = x \mmul z \Implies y = z\;.
\end{ldispl}%
In the case of a inversive meadow, the cancellation axiom is equivalent 
to the \emph{general inverse law} 
\begin{ldispl}
x \neq 0 \Implies x \mmul x\minv = 1\;.
\end{ldispl}

A \emph{totalized field} is a total algebra over the signature $\sigimd$ 
that satisfies the equations $\eqnscr$, the separation axiom, and the 
general inverse law.
A \emph{zero-totalized field} is a totalized field that satisfies in 
addition the equation $0\minv = 0$.
\begin{theorem}
\label{theorem-cMd-F0}
The class of all non-trivial inversive cancellation meadows and the 
class of all zero-totalized fields are the same.
\end{theorem}
\begin{proof}
This is a corollary of Lemma~2.5 from~\cite{BT07a}.
\qed
\end{proof}

Not all non-trivial inversive meadows are zero-totalized fields, e.g.\ 
the initial inversive meadow is not a zero-totalized field.
Nevertheless, we have the following theorem.
\begin{theorem}
\label{theorem-Md-F0}
The equational theory of inversive meadows and the equational theory of 
zero-totalized fields are the same.
\end{theorem}
\begin{proof}
This is Theorem~3.10 from~\cite{BHT09a}.
\qed
\end{proof}
Theorem~\ref{theorem-Md-F0} can be read as follows: $\eqnsimd$ is a 
finite basis for the equational theory of inversive cancellation 
meadows.

The inversive cancellation meadow that we are most interested in is 
$\Ratzi$, the zero-totalized field of rational numbers.
$\Ratzi$ differs from the field of rational numbers only in that the
multiplicative inverse of zero is zero.%
\begin{theorem}
\label{theorem-Ratzi}
$\Ratzi$ is the initial algebra among the total algebras over the 
signature $\sigimd$ that satisfy the equations
\begin{ldispl}
\eqnsimd \union \set{\ul{n} \mmul \ul{n}\minv) = 1 \where n \in \Natpos}
\quad\footnotemark
\end{ldispl}%
\footnotetext
{We write $\Natpos$ for the set $\set{n \in \Nat \where n \neq 0}$ of 
 positive natural numbers.}%
or, equivalently, the equations
\begin{ldispl}
\eqnsimd \union \set{(1 + x^2 + y^2) \mmul (1 + x^2 + y^2)\minv = 1}\;.
\end{ldispl}%
\end{theorem}
\begin{proof}
As for the first set of equations, this is Theorem~3.1 in~\cite{BT07a}.
As for the second set of equations, this is Theorem~9 in~\cite{BM09g}. 
\qed
\end{proof}

The division operator can be explicitly defined in terms of the 
multiplicative inverse operator by the equation 
$x \mdiv y = x \mmul y\minv$ and the multiplicative inverse operator can 
be explicitly defined in terms of the division operator by the equation 
$x\minv = 1 \mdiv x$.
In fact, $\eqnsimd$ and $\eqnsdmd$ are essentially the same in the sense 
which is made precise in the following theorem.
\begin{theorem}
\label{theorem-defeqv-iMd-dMd}
$\eqnsimd$ is definitionally equivalent to $\eqnsdmd$,%
\footnote
{The notion of definitional equivalence originates from~\cite{Bou65a},
 where it was introduced, in the setting of first-order theories, under 
 the name of synonymy.
 In~\cite{Tay79a}, the notion of definitional equivalence was introduced
 in the setting of equational theories under the ambiguous name of
 equivalence.
 An abridged version of~\cite{Tay79a} appears in~\cite{Gra08a}.}
i.e.
\begin{ldispl}
\eqnsimd \union \set{x \mdiv y = x \mmul y\minv}
 \vdash
\eqnsdmd \union \set{x\minv = 1 \mdiv x}
\\
\hfill \mathrm{and} \hfill \phantom{\,}
\\
\eqnsdmd \union \set{x\minv = 1 \mdiv x}
 \vdash
\eqnsimd \union \set{x \mdiv y = x \mmul y\minv}\;.
\end{ldispl}%
\end{theorem}
\begin{proof}
Because $\eqnsinv$ and $\eqnsdiv$ have $\eqnscr$ in common, in one 
direction, we only have to prove the derivability of 
$\eqnsdiv \union \set{x\minv = 1 \mdiv x}$ 
and, in the other direction, we only have to prove the derivability of 
$\eqnsinv \union \set{x \mdiv y = x \mmul y\minv}$.
The derivability of all equations involved is trivial.
\qed
\end{proof}

By the definitional equivalence of $\eqnsimd$ and $\eqnsdmd$, we have:
\begin{enumerate}
\item[(a)]
there exist a mapping $\epsilon$ from the set of all equations over 
$\sigdmd$ to the set of all equations over $\sigimd$ and a mapping 
$\epsilon'$ from the set of all equations over $\sigimd$ to the set of 
all equations over $\sigdmd$ such that, for each equation $\phi$ over 
$\sigdmd$ and each equation $\phi'$ over $\sigimd$,
\begin{ldispl}
\begin{geqns}
\eqnsimd \vdash \epsilon(\phi)  \;\;\mathrm{if}\;\; \eqnsdmd \vdash \phi\;,
\\
\eqnsdmd \vdash \epsilon'(\epsilon(\phi)) \Iff \phi\;,
\end{geqns}
\quad
\begin{geqns}
\eqnsdmd \vdash \epsilon'(\phi') \;\;\mathrm{if}\;\; \eqnsimd \vdash \phi'\;,
\\
\eqnsimd \vdash \epsilon(\epsilon'(\phi')) \Iff \phi'\;;
\end{geqns}
\end{ldispl}%
\item[(b)]
there exist a mapping $\alpha$ from the class of all divisive meadows to 
the class of all inversive meadows that maps each divisive meadow $\cM'$ 
to the restriction to $\sigimd$ of the unique expansion of $\cM'$ for 
which $\eqnsdmd \union \set{x\minv = 1 \mdiv x}$ holds 
and a mapping $\alpha'$ from the class of all inversive meadows to the 
class of all divisive meadows that maps each inversive meadow $\cM$ to 
the restriction to $\sigdmd$ of the unique expansion of $\cM$ for which 
$\eqnsimd \union \set{x \mdiv y = x \mmul y\minv}$ holds such that 
$\alpha \circ \alpha'$ and $\alpha' \circ \alpha$ are identity mappings.
\end{enumerate}
Let $\epsilon$ and $\alpha$ be as under~(a) and~(b) above.
Then it follows that, for each equation $\phi$ over $\sigdmd$ and 
divisive meadow $\cM$:   
\begin{itemize}
\item 
$\eqnsdmd \vdash \phi$ iff $\eqnsimd \vdash \epsilon(\phi)$;
\item
$\cM \models \phi$ iff $\alpha(\cM) \models \epsilon(\phi)$.
\end{itemize}
\sloppy
From many results about inversive meadows (including the ones presented
above), counterparts about divisive meadows follow immediately using 
these consequences of the definitional equivalence of $\eqnsimd$ and 
$\eqnsdmd$.
This is the main reason why the survey given in this section is not 
restricted to divisive meadows.

Henceforth, ``meadow'' without ``inversive'' or ``divisive'' as 
qualifier stands for ``inversive or divisive meadow''.

A meadow is \emph{finite} if its carrier is finite and it is 
\emph{infinite} if its carrier is infinite. 
A meadow $\cM$ is a (\emph{proper}) \emph{submeadow} of a meadow $\cM'$ 
if $\cM$ is a (proper) subalgebra of $\cM'$.
A meadow is \emph{minimal} if it does not have a proper submeadow. 
The \emph{characteristic} of a meadow is the smallest $k \in \Natpos$ 
for which it satisfies $\ul{k} = 0$.
A meadow is said to have characteristic $0$ if there does not exist an 
$k \in \Natpos$ for which it satisfies $\ul{k} = 0$.
A $k \in \Natpos$ is called \emph{square-free} if it is the product of 
distinct prime numbers.

If a meadow has characteristic $0$, then it also satisfies, for all 
$k,k' \in \Nat$, $\ul{k} \neq \ul{k'}$ if $k \neq k'$.
Hence, a meadow has characteristic $0$ only if its minimal submeadow is 
infinite.

The infinite divisive meadow that we are most interested in is $\Ratzd$, 
the zero-totalized field of rational numbers with the multiplicative 
inverse operation replaced by a division operation.
It follows immediately from Theorems~\ref{theorem-Ratzi} 
and~\ref{theorem-defeqv-iMd-dMd} that $\Ratzd$ is the initial algebra 
among the total algebras over the signature $\sigdmd$ that satisfy the 
equations 
$\eqnsdmd \union
 \set{\ul{n} \mdiv \ul{n} = 1 \where \ul{n} \in \Natpos}$.

The finite divisive meadows that we are most interested in are, for each 
square-free $k \in \Natpos$, $\Mdd{k}$, the initial algebra among the 
total algebras over the signature $\sigdmd$ that satisfy the equations 
$\eqnsdmd \union \set{\ul{k} = 0}$.
It follows immediately from Lemma~4.6 in~\cite{BHT09a} and 
Theorem~\ref{theorem-defeqv-iMd-dMd} in this paper that each minimal 
divisive meadow of characteristic $k$ is isomorphic to $\Mdd{k}$.
It follows immediately from Theorem~4.4 and Lemma~4.6 in~\cite{BHT09a}, 
Corollary~3.10 in~\cite{BRS09a}, and 
Theorem~\ref{theorem-defeqv-iMd-dMd} in this paper that, if $k$ is a 
prime number, $\Mdd{k}$ is the zero-totalized prime field of 
characteristic $k$ with the multiplicative inverse operation replaced 
by a division operation.

Several results established in 
Sections~\ref{sect-auxiliary-results}--\ref{sect-closed-fractions} 
are exclusively concerned with the divisive meadow $\Ratzd$.
The divisive meadows $\Mdd{k}$ play a role in various proofs given in 
Sections~\ref{sect-fractions}--\ref{sect-miscellaneous}.

\section{Definitions Concerning Fractions and Polynomials}
\label{sect-basic-defs-fractions}

In this section, we give the definitions concerning fractions on which 
the subsequent sections are based.
Because polynomials play a role in those sections as well, we also give
several definitions concerning polynomials in the setting of divisive 
meadows.

Henceforth, we will use the following convenient notational convention.
If we introduce a term $t$ as $t(x_1,\ldots,x_n)$, where 
$x_1,\ldots,x_n$ are distinct variables, this indicates that all 
variables that have occurrences in $t$ are among $x_1,\ldots,x_n$.
In the same context, $t(t_1,\ldots,t_n)$ is the term obtained by 
simultaneously replacing in $t$ all occurrences of $x_1$ by $t_1$ and 
\ldots\ and all occurrences of $x_n$ by $t_n$.

Fractions are viewed as terms over the signature of divisive meadows 
that are of a particular form.
This means that fractions may contain variables and may be interpreted 
in different divisive meadows. 
Thus, the view of fractions as terms leads to the need to make a 
distinction between fractions and closed fractions and to consider 
properties of fractions relative to a particular divisive meadow.
In this light, the definitions given below speak for themselves.

The following four definitions concern fractions by themselves:
\begin{itemize}
\item
a \emph{fraction} is a term over the signature $\sigdmd$ whose outermost 
operator is the operator $\mdiv$; 
\item
a \emph{simple fraction} is a fraction of which no proper subterm is 
a fraction; 
\item
a \emph{closed fraction} is a fraction that is a closed term;
\item
a \emph{simple closed fraction} is a simple fraction that is a closed 
term.
\end{itemize}
Let $E \supseteq \eqnsdmd$ be a set of equations over the signature 
$\sigdmd$, and let $\cM$ be a model of $E$. 
The following four definitions concern the transformation of terms into 
simple fractions:
\begin{itemize}
\item
$\cM$ \emph{admits transformation into simple fractions} if, for each 
term $p$ over the signature $\sigdmd$, there exists a simple fraction 
$q$ such that $\cM \models p = q$;
\item
$\cM$ \emph{admits transformation into simple fractions for closed terms} 
if, for each closed term $p$ over the signature $\sigdmd$, there exists 
a simple closed fraction $q$ such that $\cM \models p = q$;
\item
$E$ \emph{admits transformation into simple fractions} if, for each 
term $p$ over the signature $\sigdmd$, there exists a simple fraction 
$q$ such that $E \deriv p = q$;
\item
$E$ \emph{admits transformation into simple fractions for closed terms} 
if, for each closed term $p$ over the signature $\sigdmd$, there exists 
a simple closed fraction $q$ such that $E \deriv p = q$.
\end{itemize}
We believe that a fraction corresponds to what is usually meant by an 
algebraic fraction.
We are not sure to what extent a closed fraction corresponds to what is 
usually meant by a numerical fraction because definitions of the latter 
are mostly so imprecise that form and meaning seem to be mixed up.

Notice that, for each term $p$ over the signature $\sigdmd$, 
$\eqnsdmd \deriv p = p \mdiv 1$.
This means that each term $p$ over the signature $\sigdmd$ that is not a
fraction can be turned into a fraction in a trivial way.

In subsequent sections, we will phrase some results about the 
transformation of terms into simple fractions in which we refer in one 
way or another to polynomials.
Therefore, we also give several definitions concerning polynomials in 
the setting of divisive meadows.
A polynomial as defined below corresponds to what is usually meant by a 
(univariate) polynomial.

Let $x,y$ be variables, and let $\cM$ be a divisive meadow.  
The following definitions concern polynomials:
\begin{itemize}
\item
a \emph{polynomial} in the variable $x$ is a term over the signature 
$\sigdmd$ in which the operator $\mdiv$ does not occur and variables
other than $x$ do not occur;
\item
a \emph{root of} a polynomial $f(x)$ \emph{over} $\cM$ is an element $v$ 
of the carrier of $\cM$ such that $\cM$ satisfies the equation 
$f(x) = 0$ if the value assigned to $x$ is $v$;
\item
the \emph{polynomial function induced by} a polynomial $f(x)$ 
\emph{over} $\cM$ is the unary function $F$ on the carrier of $\cM$ such 
that, for each element $v$ of the carrier of $\cM$, $F(v)$ is the 
interpretation of $f(x)$ in $\cM$ if the value assigned to $x$ is~$v$;
\item
a polynomial $f(x)$ \emph{is $\cM$-equivalent to} a polynomial $g(y)$, 
written $f(x) \equiv_\cM g(y)$, if $\cM \models f(x) = g(x)$;
\item
a polynomial \emph{is in canonical form} if it is of the form 
$a_n \mmul x^n + \ldots + a_1 \mmul x + a_0$, 
where $a_i \in \set{\ul{m} \where m \in \Int}$ for each 
$i \in \set{1,\ldots,n}$.
\end{itemize}
The closed terms $a_i \in \set{\ul{m} \where m \in \Int}$ occurring in 
a polynomial in canonical form are called \emph{coefficients}.

It is a generally known fact that, for each polynomial $f(x)$, there 
exists a polynomial $g(x)$ in canonical form such that 
$\eqnscr \deriv f(x) = g(x)$.
This fact can be straightforwardly proved by induction on the structure 
of a polynomial.
The following definitions concerning polynomials are based on this fact:
\begin{itemize}
\item
a polynomial $f(x)$ \emph{is non-trivial over} $\cM$ if there exist a
polynomial $g(x)$ in canonical form, an 
$a \in \set{\ul{m} \where m \in \Int}$, and an $i \in \Nat$ such that 
$\eqnscr \deriv f(x) = g(x)$, $a \mmul x^i$ is a summand of $g(x)$, and
$\cM \not\models a = 0$;
\item
a polynomial $f(x)$ \emph{is constant over} $\cM$ if $f(x)$ is 
non-trivial over $\cM$ and there exists an 
$a \in \set{\ul{m} \where m \in \Int}$ such that $\cM \models f(x) = a$;
\item
the \emph{degree of} a polynomial $f(x)$ \emph{over} $\cM$ is:
\begin{itemize}
\item
if $f(x)$ is non-trivial and not constant over $\cM$, then the largest 
$i \in \Natpos$ for which there exists a polynomial $g(x)$ in canonical 
form and an $a \in \set{\ul{m} \where m \in \Int}$ such that 
$\eqnscr \deriv f(x) = g(x)$, $a \mmul x^i$ is a summand of $g(x)$, and
$\cM \not\models a = 0$;
\item
if $f(x)$ is non-trivial and constant over $\cM$, then $0$;
\item
if $f(x)$ is not non-trivial over $\cM$, then undefined.
\end{itemize}
\end{itemize}
Notice that a simple fraction is a fraction of which the two outermost 
proper subterms are polynomials.

Henceforth, we will often leave out the qualifier ``over $\cM$'' used 
in the definienda above in contexts where only one divisive meadow is 
under discussion.

\section{Auxiliary Results Concerning Divisive Meadows}
\label{sect-auxiliary-results}

In this section, we establish some results concerning divisive meadows 
that will be used in Sections~\ref{sect-fractions} 
and~\ref{sect-closed-fractions} to establish results about the 
transformation of terms into simple fractions in divisive meadows.

The following proposition, which concerns equations making the nature of 
the division operation in divisive meadows more clear, is useful in many
proofs.
\begin{proposition}
\label{prop-dMd-derivable}
The equations 
\begin{ldispl}
1 \mdiv 0 = 0\;, \quad
1 \mdiv 1 = 1\;, \quad
1 \mdiv (- x) = - (1 \mdiv x)\;, \quad
1 \mdiv (x \mmul y) = (1 \mdiv x) \mmul (1 \mdiv y)\;, \\
(x \mdiv y) \mmul (z \mdiv w) = (x \mmul z) \mdiv (y \mmul w)\;, \hfill
(x \mdiv y) \mdiv (z \mdiv w) = (x \mmul w) \mdiv (y \mmul z)
\phantom{\;,}
\end{ldispl}%
are derivable from the equations $\eqnsdmd$.
\end{proposition}
\begin{proof}
The derivability of the first four equations follow immediately from 
Proposition~\ref{prop-iMd-derivable} and the definitional equivalence of
$\eqnsimd$ and $\eqnsdmd$.
The last two equations are easily derivable using the fourth equation.
\qed
\end{proof}

The following proposition is also useful in several proofs.
\begin{proposition}
\label{prop-numerals}
For all $n,m \in \Nat$, the equations $\ul{n + m} = \ul{n} + \ul{m}$ and
$\ul{n \mmul m} = \ul{n} \mmul \ul{m}$ are derivable from the equations 
$\eqnsdmd$.
\end{proposition}
\begin{proof}
This is proved like Lemma~1 in~\cite{BM09g}.
\qed
\end{proof}

Closed terms over the signature of divisive meadows can be reduced to a
basic term.
The set $\cB$ of \emph{basic terms} over $\sigdmd$ is inductively 
defined by the following rules:
\begin{itemize}
\item
$\ul{0} \in \cB$; 
\item
if $n,m \in \Natpos$, then $\ul{n} \mdiv \ul{m} \in \cB$; 
\item
if $n,m \in \Natpos$, then $- (\ul{n} \mdiv \ul{m}) \in \cB$; 
\item
if $p,q \in \cB$, then $p + q \in \cB$. 
\end{itemize}
\begin{theorem}
\label{theorem-basic-terms-dmd}
For all closed terms $p$ over $\sigdmd$, there exists a $q \in \cB$ such
that $\eqnsdmd \deriv p = q$.
\end{theorem}
\begin{proof}
The proof is straightforward by induction on the structure of $p$. 
If $p$ is of the form $0$, $1$ or $p' + q'$, then it is trivial to show 
that there exists a $q \in \cB$ such that $\eqnsdmd \deriv p = q$. 
If $p$ is of the form $- p'$, $p' \mmul q'$ or $p' \mdiv q'$, then it 
follows immediately from the induction hypothesis and the following 
claims:
\begin{itemize}
\item
for all $p' \in \cB$, there exists a $p'' \in \cB$ such that 
$\eqnsdmd \deriv - p' = p''$;
\item
for all $p',q' \in \cB$, there exists a $p'' \in \cB$ such that 
$\eqnsdmd \deriv p' \mmul q' = p''$;
\item
for all $p',q' \in \cB$, there exists a $p'' \in \cB$ such that 
$\eqnsdmd \deriv p' \mdiv q' = p''$.
\end{itemize}
These claims are easily proved by induction on the structure of $p'$, 
using Proposition~\ref{prop-numerals}.
\qed
\end{proof}

It is well known that closed terms over $\sigdmd$ in which the operator
$\mdiv$ does not occur can be reduced to a closed term of the form 
$\ul{0}$, $\ul{n}$ or $- \ul{n}$.
\begin{proposition}
\label{prop-basic-terms-cr}
For all closed terms $p$ over $\sigcr$, $\eqnsdmd \deriv p = \ul{0}$ or
there exists an $n \in \Natpos$ such
that $\eqnsdmd \deriv p = \ul{n}$ or $\eqnsdmd \deriv p = - \ul{n}$.
\end{proposition}
\begin{proof}
The proof is along the lines of the proof of 
Theorem~\ref{theorem-basic-terms-dmd}, but simpler.
\qed
\end{proof}

Each infinite minimal divisive meadow has $\Ratzd$ as a homomorphic 
image.%
\footnote{$\Ratzd$ is the divisive meadow of rational numbers introduced 
at the end of Section~\ref{sect-Md}.}
\begin{theorem}
\label{theorem-homomorphic-image}
$\Ratzd$ is a homomorphic image of each infinite minimal divisive meadow.
\end{theorem}
\begin{proof}
In order to prove this theorem by contradiction,
assume that $\cM$ is an infinite minimal divisive meadow that does not
have $\Ratzd$ as a homomorphic image.
Then there exists a closed equation $p = q$ such that 
$\cM \models p = q$ and $\Ratzd \not\models p = q$.
Because $\eqnsdmd \union \set{p = q} \deriv p - q = 0$, there also 
exists a closed equation $p' = 0$ such that $\cM \models p' = 0$ and  
$\Ratzd \not\models p' = 0$.
Because of Theorem~\ref{theorem-basic-terms-dmd}, we may assume that the 
closed term $p'$ in such an equation is a basic term.

\sloppy
Let $p \in \cB$ be  such that $\cM \models p = 0$ and  
$\Ratzd \not\models p = 0$.
Then there exists an $n \in \Natpos$ such that
$\eqnsdmd \union \set{p = 0} \deriv \ul{n} = 0$.
We prove this by induction on the number of different subterms of $p$ of 
the form $p' + q'$ in which $p'$ or $q'$ has a subterm of the form 
$\ul{k} \mdiv \ul{l}$ with $k,l \in \Natpos$ and $l \neq 1$.
The basis step is easily proved: 
$\eqnsdmd \union \set{p = 0} \deriv \ul{n} = 0$ follows immediately from 
the fact that $p$ is of the form $\ul{k} \mdiv \ul{1}$ with 
$k \in \Natpos$.
The inductive step is proved in the following way.
Necessarily, $p$ is of the form $C[q + q']$ where 
$q = \ul{k} \mdiv \ul{l}$ or $q = - (\ul{k} \mdiv \ul{l})$ for some 
$k,l \in \Natpos$, and 
$q' = \ul{k'} \mdiv \ul{l'}$ or $q' = - (\ul{k'} \mdiv \ul{l'})$ for 
some $k',l' \in \Natpos$.
We only consider the case that $q = \ul{k} \mdiv \ul{l}$ and 
$q' = \ul{k'} \mdiv \ul{l'}$.
The other cases are proved analogously. 
We have 
$\eqnsdmd \union \set{p = 0} \deriv 
 (\ul{l} \mdiv \ul{l}) \mmul (\ul{l'} \mdiv \ul{l'}) \mmul 
 C[\ul{k} \mdiv \ul{l} + \ul{k'} \mdiv \ul{l'}] = 0$. 
It is easily proved that 
$\eqnsdmd \deriv (r \mdiv r) \mmul C[s] =
 (r \mdiv r) \mmul C[(r \mdiv r) \mmul s]$
for all terms $r$ and $s$ and contexts $C[\;]$ over $\sigdmd$ 
(cf.\ the proof of Corollary 3.1 in~\cite{BBP13a}).
Hence, 
$\eqnsdmd \union \set{p = 0} \deriv 
 (\ul{l} \mdiv \ul{l}) \mmul (\ul{l'} \mdiv \ul{l'}) \mmul 
 C[(\ul{l'} \mdiv \ul{l'}) \mmul (\ul{k} \mdiv \ul{l}) + 
   (\ul{l} \mdiv \ul{l}) \mmul (\ul{k'} \mdiv \ul{l'})] = 0$
and so
$\eqnsdmd \union \set{p = 0} \deriv 
 (\ul{l} \mdiv \ul{l}) \mmul (\ul{l'} \mdiv \ul{l'}) \mmul 
 C[(\ul{k} \mmul \ul{l'} + \ul{k'} \mmul \ul{l}) \mdiv 
   (\ul{l} \mmul \ul{l'})] = 0$.
From this, it follows that
$\cM \models 
 (\ul{l} \mdiv \ul{l}) \mmul (\ul{l'} \mdiv \ul{l'}) \mmul 
 C[(\ul{k} \mmul \ul{l'} + \ul{k'} \mmul \ul{l}) \mdiv 
   (\ul{l} \mmul \ul{l'})] = 0$,  
whereas
$\Ratzd \not\models 
 (\ul{l} \mdiv \ul{l}) \mmul (\ul{l'} \mdiv \ul{l'}) \mmul 
 C[(\ul{k} \mmul \ul{l'} + \ul{k'} \mmul \ul{l}) \mdiv 
   (\ul{l} \mmul \ul{l'})] = 0$
since $\Ratzd \not\models p = 0$, 
$\Ratzd \models \ul{l} \mdiv \ul{l} = 1$, and
$\Ratzd \models \ul{l'} \mdiv \ul{l'} = 1$.
In the term
$(\ul{l} \mdiv \ul{l}) \mmul (\ul{l'} \mdiv \ul{l'}) \mmul 
 C[(\ul{k} \mmul \ul{l'} + \ul{k'} \mmul \ul{l}) \mdiv 
   (\ul{l} \mmul \ul{l'})]$,
the number used for the induction is one less than in $p$.
This means that we can apply the induction hypothesis and from that it 
follows that there exists an $n \in \Natpos$ such that
$\eqnsdmd \union
 \set{(\ul{l} \mdiv \ul{l}) \mmul (\ul{l'} \mdiv \ul{l'}) \mmul 
      C[(\ul{k} \mmul \ul{l'} + \ul{k'} \mmul \ul{l}) \mdiv 
      (\ul{l} \mmul \ul{l'})] = 0} \deriv \ul{n} = 0$.

Let $n \in \Natpos$ be such that 
$\eqnsdmd \union \set{p = 0} \deriv \ul{n} = 0$.
Then it follows immediately that 
$\eqnsdmd \union \set{p = 0} \deriv \ul{m} = 0$, where $m$ is the
product of the prime factors of $n$ (each with multiplicity $1$).
Hence, each model of $\eqnsdmd \union \set{p = 0}$ has a non-zero
square-free characteristic.
Now it is an immediate corollary of Theorem~4.4 and Lemma~4.6 
in~\cite{BHT09a} that each minimal inversive meadow of non-zero 
square-free characteristic is finite.
This finiteness result carries over to minimal divisive meadows of 
non-zero square-free characteristic by 
Theorem~\ref{theorem-defeqv-iMd-dMd}.
Consequently, each minimal model of $\eqnsdmd \union \set{p = 0}$ is 
finite.
This contradicts the assumed infinity of $\cM$.
Hence, there does not exist an infinite minimal divisive meadow that
does not have $\Ratzd$ as a homomorphic image.
\qed
\end{proof}

\section{Transformation of Fractions, the General Case}
\label{sect-fractions}

In this section, we establish results about the transformation into 
simple fractions for terms over the signature of divisive meadows.
The results concerned are not restricted to closed terms.
In Section~\ref{sect-closed-fractions}, we will establish results that 
are restricted to closed terms.

The first result concerns finite divisive meadow.
\begin{theorem}
\label{theorem-finite}
Every finite divisive meadow admits transformation into simple 
fractions.
\end{theorem}
\begin{proof}
Let $\cM$ be a finite divisive meadow.
Then there exist $n,m \in \Nat$ such that $\cM \models x^n = x^m$ 
because there exist only finitely many polynomial functions induced by  
polynomials of the form $x^k$.

Let $n,m \in \Natpos$ with $n > m$ be such that $\cM \models x^n = x^m$.
Then $\eqnsdmd \union \set{x^n = x^m} \deriv 
      1 \mdiv x = x^{2 \mmul (n - m) - 1}$.
We prove this using that, for each $k \in \Nat$, 
$\eqnsdmd \deriv x = x^{k+1} \mdiv x^k$ and 
$\eqnsdmd \deriv 1 \mdiv x = x^k \mdiv x^{k+1}$ as follows:
$1 \mdiv x = x^n \mdiv x^{n+1} = x^{2 \mmul n - n} \mdiv x^{n+1} = 
 x^{2 \mmul n - m} \mdiv x^{m+1} = x^{2 \mmul (n - m) - 1}$.

From $\cM \models x^n = x^m$ and
$\eqnsdmd \union \set{x^n = x^m} \deriv 
 1 \mdiv x = x^{2 \mmul (n - m) - 1}$,
it follows that $\cM \models 1 \mdiv x = x^{2 \mmul (n - m) - 1}$.
Consequently, $\cM$ admits that each term $p$ over the signature 
$\sigdmd$ is transformed into a simple fraction. 
\qed
\end{proof}

The next two results tell us that some but not all infinite divisive 
meadows admit transformation into simple fractions.
\begin{theorem}
\label{theorem-infinite}
There exists an infinite divisive meadow that admits transformation into 
simple fractions.
\end{theorem}
\begin{proof}
Let $c_i$ be a constant for each $i \in \Nat$,
let $\sigdmdc = \sigdmd \union \set{c_i \where i \in \Nat}$, 
let $\eqnsdmdii = \eqnsdmd \union \set{\ul{2} = 0, x^2 = x}$, and
let $\cI$ be the initial algebra among the total algebras over the 
signature $\sigdmdc$ that satisfy the equations $\eqnsdmdii$.
\pagebreak[2]
Then $\cI \models 1 \mdiv x = x$.
We prove this using that $\eqnsdmd \deriv x = x^2 \mdiv x$ and 
$\eqnsdmd \deriv 1 \mdiv x = x \mdiv x^2$ as follows:
$1 \mdiv x = x \mdiv x^2 = x \mdiv x = x^2 \mdiv x = x$.

Because $\cI \models 1 \mdiv x = x$, $\cI$ admits that each term $p$ 
over the signature $\sigdmdc$ is transformed into a simple fraction.
Moreover, $\cI$ is an infinite divisive meadow.
Take two arbitrary constants $c_i$ and $c_j$ with $i \neq j$.
We find a total algebra over the signature $\sigdmdc$ that satisfies the 
equations $\eqnsdmdii$ but not the equation $c_i = c_j$ by taking the
interpretation of $0$ as the interpretation of $c_i$, taking the
interpretation of $1$ as the interpretation of $c_j$, and taking the
interpretation of either $0$ or $1$ as the interpretation of the other
constants.
If follows from the existence of such an algebra that 
$\eqnsdmdii \nderiv c_i = c_j$.
Consequently, $\cI \not\models c_i = c_j$.
Hence, $\cI$ is an infinite divisive meadow.
\qed
\end{proof}

\begin{theorem}
\label{theorem-Q0}
$\Ratzd$ does not admit transformation into simple fractions.
\end{theorem}
\begin{proof}
In order to prove this theorem by contradiction, assume that $\Ratzd$ 
admits transformation into simple fractions.
Then there exist polynomials $f(x)$ and $g(x)$ such that
$\Ratzd \models 1 + 1 \mdiv x = f(x) \mdiv g(x)$.

Let $f(x)$ and $g(x)$ be polynomials such that
$\Ratzd \models 1 + 1 \mdiv x = f(x) \mdiv g(x)$.
Then $\Ratzd \models g(0) \neq 0$ because 
$\Ratzd \models 1 = f(0) \mdiv g(0)$. 
From this, it follows that $f(x) \mdiv g(x)$,
interpreted as a real function, is continuous on a closed interval 
$[0,\epsilon]$ for $\epsilon > 0$ so small that no root of $g(x)$ is in 
this interval.
Now let $a$ be the maximal absolute value of $f(x)$ on the interval
$[0,\epsilon]$ and let $b$ be the minimal absolute value of $g(x)$ on 
this interval.
Then $f(x) \mdiv g(x) \leq a \mdiv b$ on the interval $[0,\epsilon]$.

Clearly, there exists a positive rational number $q$ in the interval 
$(0,\epsilon)$ so small that $1 + 1 \mdiv q > a \mdiv b$.
Let $q$ be a positive rational number $q$ such that 
$1 + 1 \mdiv q > a \mdiv b$. 
Then also $f(q) \mdiv g(q) > a \mdiv b$ because of the assumption that
$\Ratzd \models 1 + 1 \mdiv x = f(x) \mdiv g(x)$.
This contradicts the fact that $f(x) \mdiv g(x) \leq a \mdiv b$ on the 
interval $(0,\epsilon)$.
Hence, $\Ratzd$ does not admit transformation into simple fractions.
\qed
\end{proof}

Theorems~\ref{theorem-homomorphic-image} and~\ref{theorem-Q0} give rise 
to several corollaries.
\begin{corollary}
\label{corollary-char-zero}
A divisive meadow of characteristic $0$ does not admit transformation 
into simple fractions.
\end{corollary}
\begin{proof}
This follows immediately from Theorems~\ref{theorem-homomorphic-image}
and~\ref{theorem-Q0}, and the fact that a meadow has characteristic 
$0$ only if its minimal submeadow is infinite (see the remark 
immediately after the definition of characteristic in 
Section~\ref{sect-Md}).
\qed
\end{proof}

\begin{corollary}
\label{corollary-minimal}
A minimal divisive meadow admits transformation into simple fractions 
if and only if it is finite.
\end{corollary}
\begin{proof}
This follows immediately from Theorems~\ref{theorem-homomorphic-image}, 
\ref{theorem-finite}, and~\ref{theorem-Q0}.
\qed
\end{proof}

The next result is one in which admitting transformation into simple 
fractions is related to the existence of a polynomial with a particular 
property.
\begin{theorem}
\label{theorem-polynomial}
A divisive meadow $\cM$ admits transformation into simple fractions 
only if there exists a non-trivial polynomial $f(x)$ such that each 
element of the carrier of $\cM$ is a root of $f(x)$.
\end{theorem}
\begin{proof}
Let $\cM$ be a divisive meadow that admits transformation into simple 
fractions.
Then there exist polynomials $f(x)$ and $g(x)$ such that
$\cM \models 1 + 1 \mdiv x = f(x) \mdiv g(x)$.

Let $f(x)$ and $g(x)$ be polynomials such that
$\cM \models 1 + 1 \mdiv x = f(x) \mdiv g(x)$.
Substitution of $x$ by $0$ yields $\cM \models 1 = f(0) \mdiv g(0)$, 
so $\cM \not\models g(0) = 0$.
Multiplication by $x^2 \mmul g(x)^2$ yields
$\cM \models 
 x^2 \mmul g(x)^2 + (1 \mdiv x) \mmul x^2  \mmul  g(x)^2 =  
 (f(x) \mdiv g(x)) \mmul x^2 \mmul g(x)^2$, 
and then applying the axiom $(x \mmul x) \mdiv x = x$ twice yields
$\cM \models 
 x^2 \mmul g(x)^2 + x \mmul g(x)^2 = f(x) \mmul x^2 \mmul g(x)$. 
Hence,
$\cM \models 
 x^2 \mmul g(x)^2 + x \mmul g(x)^2 - f(x) \mmul x^2 \mmul g(x) = 0$. 
This means that each element of the carrier of $\cM$ is the root of the 
polynomial 
$x^2 \mmul g(x)^2 + x \mmul g(x)^2 - f(x) \mmul x^2 \mmul g(x)$.
It remains to be proved that this polynomial is non-trivial.

Considering that
$\eqnscr \deriv 
 x^2 \mmul g(x)^2 + x \mmul g(x)^2 - f(x) \mmul x^2 \mmul g(x) =
 g(x)^2 \mmul x + \linebreak[2] (g(x)^2 - f(x) \mmul g(x)) \mmul x^2$,
it must be the case that a canonical form of  
$x^2 \mmul g(x)^2 + x \mmul g(x)^2 - f(x) \mmul x^2 \mmul g(x)$
has the constant term $0$ and the linear term $g(0)^2 \mmul x$ as 
summands.
From $\cM \not\models g(0) = 0$, it follows that 
$\cM \not\models g(0)^2 = 0$.
Hence, $x^2 \mmul g(x)^2 + x \mmul g(x)^2 - f(x) \mmul x^2 \mmul g(x)$
is non-trivial.
\qed
\end{proof}

Theorem~\ref{theorem-polynomial} gives rise to several corollaries.
\begin{corollary}
\label{corollary-theorem-polynomial-1}
A divisive meadow whose carrier contains an element that is not the root 
of a non-trivial polynomial does not admit transformation into simple 
fractions.
\end{corollary}
\begin{proof}
This follows immediately from Theorem~\ref{theorem-polynomial}.
\qed
\end{proof}

\begin{corollary}
\label{corollary-theorem-polynomial-2}
A divisive meadow admits transformation into simple fractions only if 
there exists an $n \in \Nat$ such that each element of its carrier is 
the root of a non-trivial polynomial of degree $n$ or less.
\end{corollary}
\begin{proof}
This follows immediately from Theorem~\ref{theorem-polynomial}.
\qed
\end{proof}

The next result gives us sufficient and necessary conditions of 
admitting transformation into simple fractions for divisive meadows of 
prime characteristic.
\begin{theorem}
\label{theorem-char-non-zero}
Let $k \in \Nat$ be a prime number, and
let $\cM$ be a divisive meadow of characteristic $k$.
Then the following are equivalent:
\begin{list}{}
 {\setlength{\leftmargin}{2em} \settowidth{\labelwidth}{(iii)}}
\item[\textup{(i)}]
$\cM$ admits transformation into simple fractions;
\item[\textup{(ii)}]
there exists an $n \in \Nat$ such that each element of the carrier of 
$\cM$ is the root of a non-trivial polynomial of degree $n$ or less;
\item[\textup{(iii)}]
there exists a non-trivial polynomial $f(x)$ such that each element of 
the carrier of $\cM$ is a root of $f(x)$;
\item[\textup{(iv)}]
there exists an $n \in \Natpos$ such that 
$\cM \models 1 \mdiv x = x^n$.
\end{list}
\end{theorem}
\begin{proof}
Assume~(i).
Then~(ii) follows immediately from 
Corollary~\ref{corollary-theorem-polynomial-2}.

Assume~(ii).
Let $n \in \Nat$ be such that each element of the carrier of 
$\cM$ is the root of a non-trivial polynomial of degree $n$ or less.
\pagebreak[2]
We know from Theorem~4.4 and Lemma~4.6 in~\cite{BHT09a} and 
Theorem~\ref{theorem-defeqv-iMd-dMd} in this paper that a minimal
divisive meadow of prime characteristic is a finite divisive meadow.
Because $\cM$ is of prime characteristic and the interpretation of  
coefficients in $\cM$ is the same as the interpretation of coefficients 
in the minimal divisive submeadow of $\cM$, the set of non-trivial 
polynomials of degree $n$ or less modulo $\cM$-equivalence is finite.
Let $f(x)$ be the product of the elements of a transversal for this set.   
Then $f(x)$ is a non-trivial polynomial and each element of the carrier 
of $\cM$ is a root of~$f(x)$.

Assume~(iii).
Let $f(x)$ be a non-trivial polynomial such that each element of the 
carrier of $\cM$ is a root of $f(x)$.
Then $\cM \models f(x) = 0$.
Assume that $f(x) = a_n \mmul x^n + \ldots + a_1 \mmul x + a_0$ with 
$\cM \models a_n \neq 0$.
We know from Theorem~4.4 and Lemma~4.6 in~\cite{BHT09a}, Corollary~3.10 
in~\cite{BRS09a}, and Theorem~\ref{theorem-defeqv-iMd-dMd} in this paper 
that a minimal divisive meadow of prime characteristic is a finite 
divisive cancellation meadow.
Because $\cM$ is of prime characteristic, the interpretation of 
coefficients in $\cM$ is the same as the interpretation of coefficients 
in the minimal divisive submeadow of $\cM$, and 
$\cM \models a_n \neq 0$, we have $\cM \models a_n \mdiv a_n = 1$.
Dividing both sides of the equation $f(x) = 0$ by $a_n$ yields
$\cM \models 
 x^n + \ldots + (a_1 \mdiv a_n) \mmul x + a_0 \mdiv a_n = 0$.
From this it follows by induction on $i$ that, for all $i \geq n$, there 
exists a polynomial $g(x)$ of degree less than $n$ such that 
$\cM \models x^i = g(x)$.
From this and the fact that the set of polynomials of degree $n - 1$ or 
less modulo $\cM$-equivalence is finite (as explained above), the 
polynomials from the sequence $x$, $x^2$, $x^3$, $\ldots$ cannot be all
different modulo $\cM$-equivalence.
In other words, there exist $l,m \in \Natpos$ with $l > m$ such that 
$\cM \models x^l = x^m$. 
Let $l,m \in \Natpos$ be such that $l > m$ and $\cM \models x^l = x^m$.
Then we can prove like in the proof of Theorem~\ref{theorem-finite}
that $\cM \models 1 \mdiv x = x^{2 \mmul (l - m) - 1}$.

Assume~(iv).
Let $n \in \Natpos$ be such that $\cM \models 1 \mdiv x = x^n$.
Then it follows immediately that, for each term $p$ over $\sigdmd$, 
there exists a term $q$ over $\sigdmd$ of which no subterm is a fraction 
such that $\cM \models p = q$.
Hence, $\cM$ admits transformation into simple fractions.
\qed
\end{proof}

The next theorem tells us that admitting transformation into simple 
fractions is a property of divisive meadows that cannot be expressed as 
a first-order theory.
\begin{theorem}
\label{theorem-elementary}
Admitting transformation into simple fractions is not an elementary 
property of divisive meadows.
\end{theorem}
\begin{proof}
In order to prove this theorem by contradiction, assume that there 
exists a first-order theory $T$ over $\sigdmd$ such that the models of 
$\eqnsdmd \union T$ are precisely the divisive meadows with the 
mentioned property.

Let $T$ be a first-order theory over $\sigdmd$ such that the models of 
$\eqnsdmd \union T$ are precisely the divisive meadows with the 
mentioned property.
Because of Theorem~\ref{theorem-finite}, for each $k' \in \Nat$, there 
exists a square-free $k'' \in \Nat$ with $k'' > k'$ such that
$\Mdd{k''} \models 
 \eqnsdmd \union T \union 
 \set{\ul{k} \neq 0 \where k \in \Natpos, k \leq k'}$.%
\footnote
{$\Mdd{k}$, for square-free $k \in \Natpos$, was introduced at the end 
 of Section~\ref{sect-Md}.}
From this, it follows that, for 
each $k' \in \Nat$, there exists a divisive meadow $\cM$ such that
$\cM \models
 \eqnsdmd \union T \union \linebreak[2] 
 \set{\ul{k} \neq 0 \where k \in \Natpos, k \leq k'}$. 
Hence, by the compactness of first-order logic, there exists a divisive
meadow $\cM'$ such that
$\cM' \models
 \eqnsdmd \union T \union \set{\ul{k} \neq 0 \where k \in \Natpos}$.
In other words, there exists a divisive meadow $\cM$ of characteristic 
$0$ such that $\cM \models \eqnsdmd \union T$.
This contradicts Corollary~\ref{corollary-char-zero}.
Hence, there does not exist a first-order theory $T$ over $\sigdmd$ such 
that the models of $\eqnsdmd \union T$ are precisely the divisive 
meadows with the mentioned property.
\qed
\end{proof}

The next theorem tells us that each divisive meadow admits 
transformation into sums of simple fractions.
\begin{theorem}
\label{theorem-sums-of-fractions}
For each term $p$ over the signature $\sigdmd$, there exists a finite 
number of simple fractions, say $q_1,\ldots,q_n$, such that 
$\eqnsdmd \deriv p = q_1 + \ldots + q_n$.
\end{theorem}
\begin{proof}
By Theorem~2.1 in~\cite{BBP13a}, for each term $p$ over the signature 
$\sigimd$, there exists a term $q$ in standard meadow form such that 
$\eqnsimd \deriv p = q$.
By the distributivity of multiplication over addition, each term in 
standard meadow form is derivably equal to a sum of terms in standard 
meadow form of level $0$; and terms in standard meadow form of level $0$ 
are of the form $p' \mmul {q'}\minv$ with $p'$ and $q'$ polynomials.
By Theorem~\ref{theorem-defeqv-iMd-dMd}, this proves the current 
theorem. 
\qed
\end{proof}

A question arising from Theorem~\ref{theorem-sums-of-fractions} is 
whether there exists a natural number $k$ such that each term over the 
signature $\sigdmd$ is derivably equal to a sum of at most $k$ simple 
fractions.
It is a corollary of Theorem~\ref{theorem-bounded-sums-of-fractions} 
below that this question must be answered negatively.
Below, we will write $\Realzd$ for the zero-totalized field of real 
numbers with the multiplicative inverse operation replaced by a division 
operation.
\begin{theorem}
\label{theorem-bounded-sums-of-fractions}
For each $n \in \Natpos$, 
let $p^n(x_1,\ldots,x_n) = 1 \mdiv x_1 + \ldots + 1 \mdiv x_n$.
Then, for each $n \in \Natpos$ with $n > 1$, there do not exist 
$n - 1$ simple fractions, say $q_1,\ldots,q_{n-1}$, such that 
$\Realzd \models p^n(x_1,\ldots,x_n) = q_1 + \ldots + q_{n-1}$.
\end{theorem}
\begin{proof}
We prove this by induction on $n$.

The basis step, where $n = 2$, is easily proved by contradiction.
Assume that there exists a simple fraction $q(x_1,x_2)$ such that 
$\Realzd \models 1 \mdiv x_1 + 1 \mdiv x_2 = q(x_1,x_2)$.
Then $\Realzd \models 1 + 1 \mdiv x_2 = q(1,x_2)$.
This contradicts the fact, following immediately from 
Theorem~\ref{theorem-Q0}, that 
$\Ratzd \not\models 1 + 1 \mdiv x_2 = q(1,x_2)$.
Hence, there does not exist a simple fraction, say $q_1$, such 
that $\Realzd \models p^2(x_1,x_2) = q_1$.

The inductive step is also proved by contradiction. 
Assume that there exist $n$ simple fractions, say 
$f_1 \mdiv g_1,\ldots,f_n \mdiv g_n$, 
such that 
$\Realzd \models 
 p^{n+1}(x_1,\ldots,x_{n+1}) = f_1 \mdiv g_1 + \ldots + f_n \mdiv g_n$.
Here and in the remainder of the proof, all variables that have 
occurrences in $f_1,g_1,\ldots,f_n,g_n$ are understood to be among 
$x_1,\ldots,x_{n+1}$.
Let $r_1,\ldots,r_n$ be real numbers, and
let $\vec{r}$ be the vector $(r_1,\ldots,r_n)$.

Let $p^{n+1}[\vec{r}]$ be the unary function on real numbers that is the 
interpretation of $p^{n+1}(x_1,\ldots,x_{n+1})$ in $\Realzd$ if 
$x_1,\ldots,x_n$ are assigned the values $r_1,\ldots,r_n$, respectively.
Then $p^{n+1}[\vec{r}](v) = u + 1 \mdiv v$ with 
$u = 1 \mdiv r_1 + \ldots + 1 \mdiv r_n$ a real number.
Consequently, $p^{n+1}[\vec{r}]$ is continuous everywhere except at $0$.

For each $i \in \Natpos$ with $i \leq n$, 
let $f_i[\vec{r}]$, $g_i[\vec{r}]$, and $q_i[\vec{r}]$ be the unary 
functions on real numbers that are the interpretations of $f_i$, $g_i$,
and $f_i \mdiv g_i$, respectively, in $\Realzd$ if $x_1,\ldots,x_n$ are 
assigned the values $r_1,\ldots,r_n$, respectively.
Let $i \in \Natpos$ be such that $i \leq n$.
Then $q_i[\vec{r}](v) = f_i[\vec{r}](v) \mdiv g_i[\vec{r}](v)$.
Because $g_i$ is a polynomial, $g_i[\vec{r}]$ is continuous.
Now assume that $g_i[\vec{r}](0) \neq 0$.
Then, in the case where $g_i[\vec{r}](0) > 0$, by the continuity of 
$g_i[\vec{r}]$, there exists a rational number $\epsilon_{i,\vec{r}}$ 
such that, for each 
$v \in (-\epsilon_{i,\vec{r}},\epsilon_{i,\vec{r}})$, 
$g_i[\vec{r}](v) > 0$.
Likewise, in the case where $g_i[\vec{r}](0) < 0$, there exists a 
rational number $\epsilon_{i,\vec{r}}$ such that, for each 
$v \in (-\epsilon_{i,\vec{r}},\epsilon_{i,\vec{r}})$, 
$g_i[\vec{r}](v) < 0$.
From this, if follows that $q_i[\vec{r}]$ is continuous in the interval
$(-\epsilon_{i,\vec{r}},\epsilon_{i,\vec{r}})$ if 
$g_i[\vec{r}](0) \neq 0$.

From the continuity results established in the previous two paragraphs,
we can prove by contradiction the claim that there exists an 
$i \in \Natpos$ with $i \leq n$ such that $g_i[\vec{r}](0) = 0$.
Assume the contrary.
For each $i \in \Natpos$ with $i \leq n$, let $\epsilon_{i,\vec{r}}$ be
as indicated above.
Moreover, let $\epsilon_{\vec{r}}$ be the minimum of 
$\epsilon_{1,\vec{r}},\ldots,\epsilon_{n,\vec{r}}$.
Then, for each $i \in \Natpos$ with $i \leq n$, $q_i[\vec{r}]$ is 
continuous in the interval $(-\epsilon_{\vec{r}},\epsilon_{\vec{r}})$.
From this and the fact that 
$p^{n+1}[\vec{r}](v) = q_1[\vec{r}](v) + \ldots + q_n[\vec{r}](v)$ it 
follows that $p^{n+1}[\vec{r}]$ is continuous in the interval 
$(-\epsilon_{\vec{r}},\epsilon_{\vec{r}})$.
This contradicts the fact that $p^{n+1}[\vec{r}]$ is not continuous at $0$.
Hence, there exists an $i \in \Natpos$ with $i \leq n$ such that 
$g_i[\vec{r}](0) = 0$.

For each $i \in \Natpos$ with $i \leq n$, 
let $f'_i$ and $g'_i$ be the terms obtained by replacing in $f_i$ and 
$g_i$, respectively, all occurrences of $x_{n+1}$ by $0$, and
let $f'_i[\vec{r}]$ and $g'_i[\vec{r}]$ be the real numbers that are the 
interpretations of $f'_i$ and $g'_i$, respectively, in $\Realzd$ if 
$x_1,\ldots,x_n$ are assigned the values $r_1,\ldots,r_n$, respectively.
Then 
$\Realzd \models
 p^n(x_1,\ldots,x_n) = f'_1 \mdiv g'_1 + \linebreak[2]
 \ldots + f'_n \mdiv g'_n$
because $\Realzd \models 1 \mdiv 0 = 0$. 
Moreover, by the claim proved above, there exists an $i \in \Natpos$ 
with $i \leq n$ such that $g'_i[\vec{r}] = 0$.
From this, it follows that  
$g'_1[\vec{r}] \mmul \ldots \mmul g'_n[\vec{r}] = 0$.

Thus, we have established above that, for each vector $\vec{r}$ of $n$ 
real numbers, $g'_1[\vec{r}] \mmul \ldots \mmul g'_n[\vec{r}] = 0$.
Hence, $\Realzd \models g'_1 \mmul \ldots \mmul g'_n = 0$.
Using this result, we can prove the claim that there exists an 
$i \in \Natpos$ with $i \leq n$ such that $\Realzd \models g'_i = 0$.
For each $i \in \Natpos$ with $i \leq n$, $g'_i$ is a multivariate 
polynomial in $n$ variables.
Because the ring of polynomials over $\Real$ in $n$ variables is an 
integral domain (see e.g.\ Proposition~4.29 in~\cite{Kna06a}), the ring 
of polynomials over $\Realzd$ in $n$ variables is an integral domain as
well.
In other words, there are no zero divisors in this polynomial ring.
From this and the fact that 
$\Realzd \models g'_1 \mmul \ldots \mmul g'_n = 0$, it follows that 
there exists an $i \in \Natpos$ with $i \leq n$ such that 
$\Realzd \models g'_i = 0$.

Take $j \in \Natpos$ with $j \leq n$ such that 
$\Realzd \models g'_j = 0$.
Such a $j$ exists by the claim just proved.
Then $\Realzd \models f'_j \mdiv g'_j = 0$.
From this and the fact that 
$\Realzd \models
 p^n(x_1,\ldots,x_n) = f'_1 \mdiv g'_1 + \ldots + f'_n \mdiv g'_n$,
it follows that
$\Realzd \models
 p^n(x_1,\ldots,x_n) = 
 f'_1 \mdiv g'_1 + \ldots + f'_{j-1} \mdiv g'_{j-1} +
 f'_{j+1} \mdiv g'_{j+1} + \ldots + f'_n \mdiv g'_n$.
This contradicts the induction hypothesis.
Hence, there do not exist $n$ simple fractions, say 
$f_1 \mdiv g_1,\ldots,f_n \mdiv g_n$, 
such that 
$\Realzd \models 
 p^{n+1}(x_1,\ldots,x_{n+1}) = f_1 \mdiv g_1 + \ldots + f_n \mdiv g_n$.
\qed
\end{proof}

Theorem~\ref{theorem-bounded-sums-of-fractions} gives rise to the 
following corollary.
\begin{corollary}
\label{corollary-bounded-sums-of-fractions-1}
For each $n \in \Natpos$, 
let $p^n(x_1,\ldots,x_n) = 1 \mdiv x_1 + \ldots + 1 \mdiv x_n$.
Then, for each $n \in \Natpos$ with $n > 1$, there do not exist 
$n - 1$ simple fractions, say $q_1,\ldots,q_{n-1}$, such that 
$\eqnsdmd \deriv p^n(x_1,\ldots,x_n) = q_1 + \ldots + q_{n-1}$.
\end{corollary}
\begin{proof}
This follows immediately from 
Theorem~\ref{theorem-bounded-sums-of-fractions}.
\qed
\end{proof}

Corollary~\ref{corollary-bounded-sums-of-fractions-1} in its turn gives 
rise to the corollary announced above.
\begin{corollary}
\label{corollary-bounded-sums-of-fractions-2}
There does not exist a $k \in \Natpos$ such that, for each term $p$ over 
the signature $\sigdmd$, there exists $k$ simple fractions, say 
$q_1,\ldots,q_k$, such that $\eqnsdmd \deriv p = q_1 + \ldots + q_k$.
\end{corollary}
\begin{proof}
This follows immediately from 
Corollary~\ref{corollary-bounded-sums-of-fractions-1}.
\qed
\end{proof}

\section{Transformation of Closed Fractions}
\label{sect-closed-fractions}

In this section, we establish results about the transformation into 
simple fractions for closed terms over the signature of divisive 
meadows.

The first result concerns the axioms of a divisive meadow.
\begin{theorem}
\label{theorem-closed-Md}
$\eqnsdmd$ does not admit transformation into simple fractions for 
closed terms.
\end{theorem}
\begin{proof}
In order to prove this theorem by contradiction, assume that $\eqnsdmd$ 
admits transformation into simple fractions for closed terms.
Then, by Proposition~\ref{prop-basic-terms-cr}, 
$\eqnsdmd \deriv 1 + 1 \mdiv \ul{2} = \ul{0}$ or there exist 
$n,m \in \Natpos$ such that 
$\eqnsdmd \deriv 1 + 1 \mdiv \ul{2} = \ul{n} \mdiv \ul{m}$ or
$\eqnsdmd \deriv 1 + 1 \mdiv \ul{2} = - (\ul{n} \mdiv \ul{m})$.
However, 
$\eqnsdmd \nderiv 1 + 1 \mdiv \ul{2} = \ul{0}$ 
and, for all $n,m \in \Natpos$,
$\eqnsdmd \nderiv 1 + 1 \mdiv \ul{2} =
 - (\ul{n} \mdiv \ul{m})$
because $\Ratzd \not\models 1 + 1 \mdiv \ul{2} = \ul{0}$ and, 
for all $n,m \in \Natpos$, 
$\Ratzd \not\models 1 + 1 \mdiv \ul{2} = - (\ul{n} \mdiv \ul{m})$.
Consequently, there exist $n,m \in \Natpos$ such that 
$\eqnsdmd \deriv 1 + 1 \mdiv \ul{2} = \ul{n} \mdiv \ul{m}$.
 
Let $n,m \in \Natpos$ be such that 
$\eqnsdmd \deriv 1 + 1 \mdiv \ul{2} = \ul{n} \mdiv \ul{m}$. 
Then 
$\Ratzd \models 1 + 1 \mdiv \ul{2} = \ul{n} \mdiv \ul{m}$.
Because also $\Ratzd \models 1 + 1 \mdiv \ul{2} = \ul{3} \mdiv \ul{2}$, 
we have $\Ratzd \models \ul{n} \mdiv \ul{m} = \ul{3} \mdiv \ul{2}$.
From this, it follows that 
$\Ratzd \models \ul{n} \mmul \ul{2} = \ul{m} \mmul \ul{3}$, and 
consequently, by Proposition~\ref{prop-numerals},
$\Ratzd \models \ul{n \mmul 2} = \ul{m \mmul 3}$.
Hence, $n \mmul 2 = m \mmul 3$.
From this and the fact that $2$ and $3$ are relatively prime, $n$ is a 
multiple of $3$ and $m$ is a multiple of $2$.

Let $m' \in \Int$ be such that $m = 2 \mmul m'$.
Then $\Mdd{2} \models \ul{m} = \ul{2 \mmul m'}$, and 
consequently, by Proposition~\ref{prop-numerals}, 
$\Mdd{2} \models \ul{m} = \ul{2} \mmul \ul{m'}$.
From this and the fact that $\Mdd{2} \models \ul{2} = 0$, it follows 
that $\Mdd{2} \models \ul{m} = 0$.
From this, it follows that $\Mdd{2} \models \ul{n} \mdiv \ul{m} = 0$.
From this and the fact that $\Mdd{2} \models 1 + 1 \mdiv \ul{2} = 1$, 
it follows that 
$\Mdd{2} \not\models 1 + 1 \mdiv \ul{2} = \ul{n} \mdiv \ul{m}$.
This contradicts 
$\eqnsdmd \deriv 1 + 1 \mdiv \ul{2} = \ul{n} \mdiv \ul{m}$.
Hence, $\eqnsdmd$ does not admit transformation into simple fractions 
for closed terms.
\qed
\end{proof}

The next result and the result following the second next result tell us 
that $\Ratzd$ is the only minimal model of $\eqnsdmd$ with an infinite 
carrier that admits transformation into simple fractions for closed 
fractions.
The second next result is an auxiliary result used to establish the
result following it.
\begin{theorem}
\label{theorem-closed-Q0}
$\Ratzd$ admits transformation into simple fractions for closed terms.
\end{theorem}
\begin{proof}
Because of Theorem~\ref{theorem-basic-terms-dmd}, it suffices to prove 
that for all $p \in \cB$, there exists a simple closed fraction $q$ such 
that $\Ratzd \models p = q$.
The proof is straightforward by induction on the structure of $p$.
If $p$ is of the form $\ul{0}$, $\ul{n} \mdiv \ul{m}$ or 
$- (\ul{n} \mdiv \ul{m})$, with $n,m \in \Natpos$, then it is trivial to 
show that there exists a simple closed fraction $q$ such that 
$\Ratzd \models p = q$.
If $p$ is of the form $p' + q'$, then it follows immediately from the
induction hypothesis and the claim that, for all 
$n,m,n',m' \in \Natpos$,  
$\Ratzd \models \ul{n} \mdiv \ul{m} + \ul{n'} \mdiv \ul{m'} = 
                (\ul{n} \mmul \ul{m'} + \ul{n'} \mmul \ul{m}) \mdiv
                (\ul{m} \mmul \ul{m'})$.
This claim is easily proved using the fact that, for all 
$n \in \Natpos$, $\Ratzd \models \ul{n} \mdiv \ul{n} = 1$.
\qed
\end{proof}

\begin{lemma}
\label{lemma-consistency}
Let $E \supseteq \eqnsdmd$ be a set of equations over the signature 
$\sigdmd$ and $p$ be a closed term over the signature $\sigdmd$.
Then $E \nderiv p = 0$ only if there exists a divisive
meadow $\cM$ such that $\cM \models E \union \set{p \mdiv p = 1}$ and
$\cM \models 0 \neq 1$.
\end{lemma}
\begin{proof}
In order to prove this theorem by contraposition, assume that there does 
not exist a divisive meadow $\cM$ such that 
$\cM \models E \union  \set{p \mdiv p = 1}$ 
and $\cM \models 0 \neq 1$.
In other words, assume that for each divisive meadow $\cM$, 
$\cM \models E \union \set{p \mdiv p = 1}$ implies 
$\cM \models 0 = 1$.

Let $\cM'$ be a divisive cancellation meadow such that $\cM' \models E$.
Then $\cM' \models p \mdiv p = 1$ implies $\cM' \models 0 = 1$.
Assume that $\cM' \models p \neq 0$.
Then $\cM' \models p \mdiv p = 1$ and, because 
$\cM' \models p \mdiv p = 1$ implies $\cM' \models 0 = 1$,
also $\cM' \models 0 = 1$.
However, if $\cM' \models 0 = 1$, then $\cM' \models p = 0$.
This contradicts the assumption that $\cM' \models p \neq 0$.
Hence, $\cM' \models p = 0$.
From this, it follows that $E \deriv p = 0$ by 
Theorems~\ref{theorem-Md-F0} and~\ref{theorem-defeqv-iMd-dMd}.
\qed
\end{proof}

\begin{theorem}
\label{theorem-closed-infinite-minimal}
An infinite minimal divisive meadow admits transformation into simple 
fractions for closed terms only if it is isomorphic to $\Ratzd$.
\end{theorem}
\begin{proof}
Let $\cM$ be an infinite minimal divisive meadow that admits 
transformation into simple fractions for closed terms.
By Theorem~\ref{theorem-homomorphic-image}, $\Ratzd$ is a 
homomorphic image of $\cM$.
Let $E_\cM$ be the set of all equations of the form
$u' \mdiv v' + u'' \mdiv v'' = u \mdiv v$,
with $u$, $v$, $u'$, $v'$, $u''$, and $v''$ of the form $\ul{n}$ or
$- \ul{n}$ with $n \in \Natpos$, that are satisfied by $\cM$.

In order to prove by contradiction that $\cM$ is isomorphic to $\Ratzd$, 
assume that $\cM$ is not isomorphic to $\Ratzd$.
Then there exists a $l \in \Natpos$ such that 
$\cM \not\models  \ul{l} \mdiv \ul{l} = 1$
because $\Ratzd$ is the initial algebra of 
$\eqnsdmd \union \set{\ul{l} \mdiv \ul{l} = 1 \where l \in \Natpos}$.
Consequently, there exists a $l \in \Natpos$ such that
$\eqnsdmd \union E_\cM \nderiv \ul{l} \mdiv \ul{l} = 1$.
Let $l \in \Natpos$ be such that 
$\eqnsdmd \union E_\cM \nderiv \ul{l} \mdiv \ul{l} = 1$.
By Lemma~\ref{lemma-consistency}, 
there exists a divisive meadow $\cM'$ such that
$\cM' \models
 \eqnsdmd \union E_\cM \union 
 \set{(1 - \ul{l} \mdiv \ul{l}) \mdiv (1 - \ul{l} \mdiv \ul{l}) = 1}$
and $\cM' \models 0 \neq 1$.

Let $\cM'$ be a divisive meadow such that
$\cM' \models 
 \eqnsdmd \union E_\cM \union
 \set{(1 - \ul{l} \mdiv \ul{l}) \mdiv \linebreak[2]
      (1 - \ul{l} \mdiv \ul{l}) = 1}$
and  $\cM' \models 0 \neq 1$, and 
let $\cM''$ be the minimal divisive submeadow of $\cM'$.
Then, for each set of equations $E$ over $\sigdmd$, $\cM' \models E$ 
implies $\cM'' \models E$. 
In particular, $\cM'' \models E_\cM$. 
$\cM''$ does not have $\Ratzd$ as a homomorphic image because 
$\Ratzd \models 
 (1 - \ul{l} \mdiv \ul{l}) \mdiv (1 - \ul{l} \mdiv \ul{l}) = 0$.
From this and Theorem~\ref{theorem-homomorphic-image}, it follows that
$\cM''$ is a finite minimal divisive meadow.
We know from Lemmas~4.1 and~4.8 in~\cite{BHT09a} and 
Theorem~\ref{theorem-defeqv-iMd-dMd} in this paper that each finite 
minimal divisive meadow is isomorphic to a divisive meadow $\Mdd{k}$ 
for some $k \in \Natpos$ that is square-free.
So we may assume that $\cM'' = \Mdd{k}$ for some $k \in \Natpos$.

Let $k \in \Natpos$ be such that $\cM'' = \Mdd{k}$.
Then there exist $n,m \in \Int$ such that 
$1 + 1 \mdiv \ul{k} = \ul{n} \mdiv \ul{m} \in E_\cM$. 
Let $n,m \in \Int$ be such that 
$1 + 1 \mdiv \ul{k} = \ul{n} \mdiv \ul{m} \in E_\cM$. 
Then $\Ratzd \models 1 + 1 \mdiv \ul{k} = \ul{n} \mdiv \ul{m}$ because 
$\Ratzd$ is a homomorphic image of $\cM$.
From this, it follows that 
$\Ratzd \models (\ul{k} + 1) \mdiv \ul{k} = \ul{n} \mdiv \ul{m}$, and
consequently, by Proposition~\ref{prop-numerals},
$\Ratzd \models \ul{k + 1} \mdiv \ul{k} = \ul{n} \mdiv \ul{m}$.
From this, it follows that 
$\Ratzd \models \ul{n} \mmul \ul{k} = \ul{m} \mmul \ul{k + 1}$
and consequently, by Proposition~\ref{prop-numerals},
$\Ratzd \models \ul{n \mmul k} = \ul{m \mmul (k + 1)}$.
Hence, $n \mmul k = m \mmul (k + 1)$.
From this and the fact that $k$ and $k + 1$ are relatively prime, it
follows that $n$ is a multiple of $k + 1$ and $m$ is a multiple of $k$.
Let $m' \in \Int$ be such that $m = k \mmul m'$.
Then $\Mdd{k} \models \ul{m} = \ul{k \mmul m'}$, and consequently, by 
Proposition~\ref{prop-numerals}, 
$\Mdd{k} \models \ul{m} = \ul{k} \mmul \ul{m'}$.
From this and the fact that $\Mdd{k} \models \ul{k} = 0$, it follows 
that $\Mdd{k} \models \ul{m} = 0$.
From this, it follows that $\Mdd{k} \models \ul{n} \mdiv \ul{m} = 0$.
From this and the fact that $\Mdd{k} \models 1 + 1 \mdiv \ul{k} = 1$, it 
follows that 
$\Mdd{k} \not\models 1 + 1 \mdiv \ul{k} = \ul{n} \mdiv \ul{m}$.
Because $\Mdd{k} = \cM''$ and 
$1 + 1 \mdiv \ul{k} = \ul{n} \mdiv \ul{m} \in E_\cM$,
this contradicts $\cM'' \models E_\cM$.
Hence, $\cM$ is isomorphic to $\Ratzd$.
\qed
\end{proof}

The next theorem tells us that admitting transformation into simple 
fractions for closed terms is a property of divisive meadows that cannot 
be expressed as a first-order theory.
\begin{theorem}
\label{theorem-closed-elementary}
Admitting transformation into simple fractions for closed terms is not 
an elementary property of divisive meadows.
\end{theorem}
\begin{proof}
In order to prove this theorem by contradiction, assume that there 
exists a first-order theory $T$ over $\sigdmd$ such that the models of 
$\eqnsdmd \union T$ are precisely the divisive meadows with the 
mentioned property.
In this proof, we write $\mathit{NC}$ for 
$\Exists{x}{(x \neq 0 \And x \mdiv x \neq 1)}$.

Let $T$ be a first-order theory over $\sigdmd$ such that the models of 
$\eqnsdmd \union T$ are precisely the divisive meadows with the 
mentioned property.
Because of Theorem~\ref{theorem-finite}, for each $k' \in \Nat$, there 
exists a square-free $k'' \in \Nat$ with $k''$ not prime and $k'' > k'$ 
such that
$\Mdd{k''} \models
 \eqnsdmd \union T \union \set{\mathit{NC}} \union 
 \set{\ul{k} \neq 0 \where k \in \Natpos, k \leq k'}$.
From this, it follows that, for 
each $k' \in \Nat$, there exists a divisive meadow $\cM$ such that
$\cM \models
 \eqnsdmd \union T \union \set{\mathit{NC}} \union
 \set{\ul{k} \neq 0 \where k \in \Natpos, k \leq k'}$.
Hence, by the compactness of first-order logic, there exists a divisive
meadow $\cM'$ such that
$\cM' \models
 \eqnsdmd \union T \union \set{\mathit{NC}} \union
 \set{\ul{k} \neq 0 \where k \in \Natpos}$.
In other words, there exists a divisive meadow $\cM$ of characteristic 
$0$ such that $\cM \models \eqnsdmd \union T \union \set{\mathit{NC}}$.
This contradicts Theorem~\ref{theorem-closed-infinite-minimal} because 
a divisive meadow of characteristic $0$ is infinite and a divisive 
meadow that satisfies $\mathit{NC}$ is not a divisive cancellation 
meadow. 
Hence, there does not exist a first-order theory $T$ over $\sigdmd$ such 
that the models of $\eqnsdmd \union T$ are precisely the divisive 
meadows with the mentioned property.
\qed
\end{proof}

\section{Miscellaneous Results about Divisive Meadows}
\label{sect-miscellaneous}

In this section, we establish two results that are related to the 
results in preceding sections, but do not concern fractions.

In the proof of Theorem~\ref{theorem-closed-infinite-minimal}, one of 
the main results of this paper, Theorem~\ref{theorem-homomorphic-image}
plays an important part.
The first result, which is a generalization of the result that a 
polynomial in canonical form is derivably equal to $0$ only if each of 
its coefficients is derivably equal to $0$, is established by means of  
Theorem~\ref{theorem-homomorphic-image} as well.
\begin{theorem}
\label{theorem-polynomials}
Let $f(x)$ be a term over $\sigdmd$ of the form 
$a_n \mmul x^n + \ldots + a_1 \mmul x + a_0$, where $a_i$ is a closed 
term over $\sigdmd$ for each $i \in \set{0,\ldots,n}$.
Then $\eqnsdmd \deriv f(x) = 0$ only if, for each 
$i \in \set{0,\ldots,n}$, $\eqnsdmd \deriv a_i = 0$.
\end{theorem}
\begin{proof}
Because $\eqnsdmd \deriv f(x) = 0$, we have $\Ratzd \models f(x) = 0$.
From this and the fact that there exists an $i \in \set{0,\ldots,n}$ 
such that $\Ratzd \not\models a_i = 0$ only if 
$\Ratzd \not\models f(x) = 0$, it follows that $\Ratzd \models a_i = 0$ 
for each $i \in \set{0,\ldots,n}$.

In order to prove this theorem by contradiction, assume that there 
exists an $i \in \set{0,\ldots,n}$ such that 
$\eqnsdmd \nderiv a_i = 0$.
Substitution of $x$ by $0$ in $f(x) = 0$ yields 
$\eqnsdmd \deriv a_0 = 0$. 
Therefore, there exists an $i \in \set{1,\ldots,n}$ such that 
$\eqnsdmd \nderiv a_i = 0$.
Let $m \in \set{1,\ldots,n}$ be the maximal $i \in \set{1,\ldots,n}$ 
such that $\eqnsdmd \nderiv a_i = 0$.
Then $\eqnsdmd \deriv a_m \mmul x^m + \ldots + a_1 \mmul x = 0$.

By Lemma~\ref{lemma-consistency}, we know that there exists a divisive 
meadow $\cM$ such that 
$\cM \models \eqnsdmd \union \set{a_m \mdiv a_m = 1}$ 
and $\cM \models 0 \neq 1$.
Let $\cM$ be a divisive meadow such that
$\cM \models \eqnsdmd \union \set{a_m \mdiv a_m = 1}$
and $\cM \models 0 \neq 1$, and 
let $\cM'$ be the minimal divisive submeadow of $\cM$.
Then, for each equation $\phi$ over $\sigdmd$, $\cM \models \phi$ 
implies $\cM' \models \phi$. 
In particular, $\cM' \models a_m \mdiv a_m = 1$.

$\Ratzd$ is not a homomorphic image of $\cM'$ because otherwise 
$\Ratzd \models a_m \mdiv a_m = 1$ and consequently
$\Ratzd \models a_m \neq 0$.
From this and Theorem~\ref{theorem-homomorphic-image}, it follows that
$\cM'$ is a finite minimal divisive meadow.
We know from Lemmas~4.1 and~4.8 in~\cite{BHT09a} and 
Theorem~\ref{theorem-defeqv-iMd-dMd} in this paper that each finite 
minimal divisive meadow is isomorphic to a divisive meadow $\Mdd{k}$ for 
some $k \in \Natpos$ that is square-free.
So we may assume that $\cM' = \Mdd{k}$ for some $k \in \Natpos$ that is
square-free.
Let $k \in \Natpos$ be such that $\cM' = \Mdd{k}$.
Then 
$\eqnsdmd \union \set{\ul{k} = 0} \deriv
 a_m \mdiv a_m = 1$ 
because $\Mdd{k}$ is the initial algebra of 
$\eqnsdmd \union \set{\ul{k} = 0}$.
From this, it follows that 
$\eqnsdmd \union \set{\ul{k} = 0} \deriv 
 x^m + \ldots + (a_1 \mdiv a_m) \mmul x = 0$.
Now, let $k' \in \Natpos$ be a prime factor of $k$.
Then, 
because $\eqnsdmd \union \set{\ul{k'} = 0} \deriv \ul{k} = 0$, also 
$\eqnsdmd \union \set{\ul{k'} = 0} \deriv 
 x^m + \ldots + (a_1 \mdiv a_m) \mmul x = 0$.

From this, it follows that
$\Mdd{k'} \models x^m + \ldots + (a_1 \mdiv a_m) \mmul x = 0$, and
con\-se\-quent\-ly,
$\Mdd{k'} \models x^m + \ldots + (a_1 \mdiv a_m) \mmul x + 1 = 1$.
Because $\Mdd{k'} \not\models 1 = 0$, there exists an algebraic 
extension $\cM''$ of $\Mdd{k'}$ and an element $v$ of the carrier of 
$\cM''$ such that $\cM''$ satisfies 
$x^m + \ldots + (a_1 \mdiv a_m) \mmul x + 1 = 0$ 
if the value assigned to $x$ is $v$.
So, there exists a divisive meadow $\cM''$ such that 
$\cM'' \models \ul{k'} = 0$ 
and an element $v$ of the carrier of $\cM''$ such that $\cM''$ does not 
satisfy $x^m + \ldots + (a_1 \mdiv a_m) \mmul x = 0$ if the value 
assigned to $x$ is $v$.
This contradicts 
$\eqnsdmd \union \set{\ul{k'} = 0} \deriv 
 x^m + \ldots + (a_1 \mdiv a_m) \mmul x = 0$.
Hence, for each $i \in \set{0,\ldots,n}$, $\eqnsdmd \deriv a_i = 0$.
\qed
\end{proof}

A polynomial in canonical form is derivably equal to $0$ only if each of 
its closed substitution instances is derivably equal to $0$.
The question arises whether ``only if'' can be replaced by ``if and only 
if''.
The proof of the next theorem gives a negative answer this question.
\begin{theorem}
\label{theorem-omega-incomplete}
$\eqnsdmd$ is not $\omega$-complete.
\end{theorem}
\begin{proof}
First we prove by contradiction that 
$\eqnsdmd \nderiv
 (1 - \ul{2} \mdiv \ul{2}) \mmul (x^2 - x) = 0$ 
and next we prove that
$\eqnsdmd \deriv (1 - \ul{2} \mdiv \ul{2}) \mmul (p^2 - p) = 0$
for each closed term $p$ over~$\sigdmd$.

Assume that 
$\eqnsdmd \deriv (1 - \ul{2} \mdiv \ul{2}) \mmul (x^2 - x) = 0$.
Then 
$\eqnsdmd \union \set{\ul{2} = 0} \deriv x^2 - x = 0$, 
and consequently $\cM \models x^2 = x$ for each divisive meadow $\cM$ of 
characteristic $2$.
Let $\cM'$ be a finite algebraic extension of $\Mdd{2}$ such that 
$\cM' \models x^2 + x + 1 = 0$. 
Then $\cM'$ is a divisive meadow of characteristic $2$ such that 
$\cM' \not\models x^2 = x$.
This is easy to see: otherwise $\cM' \models x + x + 1 = 0$ and 
consequently $\cM' \models 1 = 0$.
However, $\cM' \not\models x^2 = x$ contradicts $\cM \models x^2 = x$ 
for each divisive meadow $\cM$ of characteristic $2$.
Hence, $\eqnsdmd \nderiv
 (1 - \ul{2} \mdiv \ul{2}) \mmul (x^2 - x) = 0$.

Let $\cM$ be a divisive cancellation meadow.
If $\cM \models \ul{2} = 0$, then it is easily proved by structural 
induction that either $\cM \models p = 0$ or $\cM \models p = 1$, and 
consequently $\cM \models p^2 - p = 0$, for each closed term $p$ over 
$\sigdmd$.
If $\cM \models \ul{2} \neq 0$, then 
$\cM \models 1 - \ul{2} \mdiv \ul{2} = 0$.
Hence, $\cM \models (1 - \ul{2} \mdiv \ul{2}) \mmul (p^2 - p) = 0$
for each closed term $p$ over $\sigdmd$.
From this, it follows that 
$\eqnsdmd \deriv (1 - \ul{2} \mdiv \ul{2}) \mmul (p^2 - p) = 0$
for each closed term $p$ over $\sigdmd$
by Theorems~\ref{theorem-Md-F0} and~\ref{theorem-defeqv-iMd-dMd}.
\qed
\end{proof}
The $\omega$-completeness question can also be posed about extensions of 
$\eqnsdmd$ that exclude divisive meadows of non-zero characteristic, such 
as 
$\eqnsdmd \union \linebreak
 \set{(1 + x^2 + y^2) \mdiv (1 + x^2 + y^2) = 1}$
($\Ratzd$ is the initial algebra among the divisive meadows that satisfy 
this extension of $\eqnsdmd$). 
Such variants of the question are related to Hilbert's tenth problem and 
are harder to answer.

\section{Concluding Remarks}
\label{sect-concl}

We have shown that the setting of meadows allows workable syntactic 
definitions of a fraction and a simple fraction to be given.
This only means that we have a point of departure for the development of
a workable theory about fractions.
We have made a start with the development of such a theory, but there 
remain many open questions.
For instance, it is an open question, arising from 
Theorem~\ref{theorem-char-non-zero},
whether each divisive meadow of non-zero characteristic for which 
there exists an $n \in \Nat$ such that each element of its carrier is 
the root of a non-trivial polynomial of degree $n$ or less admits 
transformation into simple fractions --- we know already from 
Theorem~\ref{theorem-char-non-zero} that there exists such an $n$ for 
each divisive meadow of prime characteristic.
Another open question, arising from the proof of 
Corollary~\ref{corollary-bounded-sums-of-fractions-2}, is whether there 
exists a natural number $k$ such that each term in one variable over 
$\sigdmd$ is derivably equal to a sum of at most $k$ simple fractions.

There are questions that are not complicated for simple fractions, but
complicated for more complex terms.
Let $\Complzd$ be the zero-totalized field of complex numbers with the 
multiplicative inverse operation replaced by a division operation.
Using Robinson's classical result that the first order theory of an 
algebraically closed field is model complete~\cite{Rob56a}, it is proved 
in~\cite{BB15b} that the equational theory of $\Complzd$ and the 
equational theory of the class of all models of 
$\eqnsdmd \union \set{\ul{n} \mdiv \ul{n} = 1 \where n \in \Natpos}$ are
the same.
From this, it follows that $\Complzd \models p = 0$ if and only if 
$\eqnsdmd \union \set{\ul{n} \mdiv \ul{n} = 1 \where n \in \Natpos}
  \deriv p = 0$.
A simple direct proof with the theory developed so far can be found if 
$p$ is restricted to simple fractions.
However, it seems less straightforward to find such a proof if $p$ is 
restricted to sums of simple fractions.

\subsection*{Acknowledgements}
We thank an anonymous referee for carefully reading a preliminary 
version of this paper and for suggesting improvements of the 
presentation of the paper.

\bibliographystyle{splncs03}
\bibliography{MD}

\end{document}